\documentclass{amsart}
\usepackage{cite}
\usepackage{amssymb,amsmath,amsfonts}
\usepackage{graphicx}
\newtheorem{theorem}{Theorem}[section]
\newtheorem{lemma}[theorem]{Lemma}
\newtheorem{cor}[theorem]{Corollary}

\theoremstyle{definition}

\newtheorem{example}[theorem]{Example}

\newtheorem{remark}[theorem]{Remark}

\numberwithin{equation}{section}
\begin{document}
\title[fixed points approximation]{the split feasibility and fixed point equality problems for quasi-nonexpansive mappings in Hilbert spaces}
\date{\today}
\author[L.B. Mohammed, A. K{\i}l{\i}\c{c}man] {L.B. Mohammed and A. K{\i}l{\i}\c{c}man}
\address{Department of Mathematics, Faculty of Science
\newline \indent \hspace{2mm} Universiti Putra Malaysia, 43400 Serdang Selangor, Malaysia.} 
\email{akilicman@science.upm.edu.my, lawanbulama@gmail.com}
\begin{abstract}
In this paper, we introduce  a new problem called the split feasibility and fixed point equality problems (SFFPEP) and 
propose a new iterative algorithm for solving the problem (SFFPEP) for the class of quasi-nonexpansive mappings in Hilbert spaces. Furthermore, we study the convergence of the proposed algorithm. At the end, we give  numerical example that illustrate our theoretical result. The SFFPEP is a generalization of the split feasibility problem (SFP),  split feasibility and fixed point problems (SFFPP) and split equality fixed point problem (SEFPP). 
\end{abstract}
\maketitle
\section{Introduction}
The split feasibility problem (SFP) in finite-dimensional Hilbert space was first introduced in 1994 by Censor and Elfving \cite{4}, this problem is useful to some area of applied mathematics, such as in convex optimization, image recovery, etc.   
Recently, it was found that the SFP can also be applied to study intensity-modulated radiation therapy; see, for example, \cite{5,6,7} and the references therein. For many years, a wide variety of iterative methods has been used to approximate the solution of SFP, for example, see \cite{9,10,11,12} and references therein.

The SFP is formulated as follows:
\begin{equation}
{\rm~Find~}x^{*}\in C {\rm~such~ that~} y^{*}\in Q,\label{1.1}
\end{equation}
where $C$ and $Q$ are nonempty closed convex subset of Hilbert space $H_{1}$ and  $H_{2},$ respectively, and $A:H_{1}\to H_{2}$ is a bounded linear operator.

The split feasibility and fixed point problems (SFFPP) is required to find a vector 
\begin{equation}
x^{*}\in C\cap Fix(U) {\rm~such~ that~} Ax^{*}\in Q\cap Fix(T),\label{1.2}
\end{equation}
where $U:H_{1}\to H_{1}$ and $T:H_{2}\to H_{2}$ are two nonlinear mappings, and $A:H_{1}\to H_{2}$ is a bounded linear operator.
  It is easy to see that  Problem (\ref{1.2}) reduces to the Problem (\ref{1.1}) as
$C:=Fix(U)$ and $Q:=Fix(T).$ Therefore, it is worth to mentioned here that  Problem (\ref{1.2})  generalizes
Problem (\ref{1.1}).

The split equality fixed point problems (SEFPP) was  introduced  by Moudafi  \cite{1} and it takes the following form:
\begin{equation}
{\rm~Find~} x^{*}\in C {\rm~and~} y^{*}\in Q {\rm~such~ that~} Ax^{*}=By^{*}.\label{1.3}
\end{equation}
where $A:H_{1}\to H_{3}$ and $B:H_{2}\to H_{3}$  are two bounded linear operators, $C$ and $Q$ be a nonempty closed convex subset of $H_{1}$ and  $H_{2},$ respectively. It is easy to see that  Problem (\ref{1.3}) reduces to  Problem (\ref{1.1}) as $H_{2}=H_{3}$ and $B=I$ ($I$ is the identity operator on $H_{2}$) in (\ref{1.3}). Therefore
 Problem (\ref{1.3}) proposed by Moudafi  \cite{1} is a generalization of  Problem (\ref{1.1}).

We now introduce a new problem called the split feasibility and fixed point equality problems (SFFPEP), this is fomulated as:
\begin{equation}
{\rm~Find~} x^{*}\in C\cap Fix(U) {\rm~and~} y^{*}\in Q\cap Fix(T) {\rm~such~ that~} Ax^{*}=By^{*},\label{1.4}
\end{equation}
where $U:H_{1}\to H_{1}$ and  $T:H_{2}\to H_{2}$ are  two quasi-nonexpansive mappings with  $Fix(U)\neq\emptyset$ and
$Fix(T)\neq\emptyset,$  $A:H_{1}\to H_{3}$ and $B:H_{2}\to H_{3}$  are two bounded linear operators, $C$ and $Q$ are two nonempty closed convex subset of $H_{1}$ and  $H_{2},$ respectively.

Note that if $C:=Fix(U)$ and $Q:=Fix(T),$ then, Problem (\ref{1.4}) reduces to Problem (\ref{1.3}) and also reduces to Problem (\ref{1.2}) as $H_{2}=H_{3}$ and $B=I$ ($I$ stands for the identity operator on $H_{2}$) in (\ref{1.4}). In the light of this, it worth to mention here that the SFFPEP generalizes the SFP,  SFFPP and SEFPP. Therefore, the results and
conclusions that are true for the SFFPEP  continue to holds for these problems (SFP, SFFPP and SEFPP)
and it  definitely shows the significance and the range of applicability of SFFPEP.

In order to approximate the solution of SEFPP (\ref{1.2}), Moudafi and Al-Shemas \cite{2} introduced the  following  simultaneous iterative methods which generate a sequences $\{x_{n}\}$ and $\{y_{n}\}$ by

\begin{equation}
 {} \left\{ \begin{array}{ll}  
x_{n+1}=U(x_{n}-\lambda_{n} A^{*}(Ax_{n}-By_{n}),\\
\\ y_{n+1}=T(y_{n}+\lambda_{n} B^{*}(Ax_{n}-By_{n}), \forall n\geq 1, & \textrm{ $  $}
 \end{array}  \right.\label{1.5}
\end{equation}
where $U:H_{1}\to H_{1},$  $T:H_{2}\to H_{2}$ are two firmly quasi-nonexpansive mappings, $A:H_{1}\to H_{3},$   $B:H_{2}\to H_{3}$ are two bounded linear operators with their adjoints  $A^{*}$ and $B^{*},$ respectively, $\lambda_{n}\subset\left(\epsilon, \frac{2}{L_{A^{*}A}L_{B^{*}B}}\right),$ $L_{A^{*}A}$ and $L_{B^{*}B}$ denote the spectral radius of the operators $A^{*}A$ and $B^{*}B,$ respectively.

Noticing  that projection operators have very attractive properties that make them
particularly well suited for iterative algorithms, for example, see \cite{3}. By setting $U=P_{C}$ and $T=P_{Q},$  where $P_{C}$ and $P_{Q}$ denote the metric projection of $H_{1}$ and $H_{2}$ onto $C$ and $Q,$ respectively.  Trivially, Algorithm (\ref{1.5}) reduces to the following simultaneous iterative method:
\begin{equation}
 {} \left\{ \begin{array}{ll} 
x_{n+1}=P_{C}(x_{n}-\lambda_{n} A^{*}(Ax_{n}-By_{n}),\\
\\ y_{n+1}=P_{Q}(y_{n}+\lambda_{n} B^{*}(Ax_{n}-By_{n}), \forall n\geq 1, & \textrm{ $  $}
 \end{array}  \right.\label{k}
\end{equation}
this algorithm was investigated in \cite{13} by means of the projected Landweber's algorithm. We already mentioned that if $B=I,$ Problem (\ref{1.3}) reduces to the classical SFP (\ref{1.1}), and if in addition, $\lambda_{n}=\lambda=1,$ the second equation of Algorithm (\ref{k}) reduces to  $y_{n+1}=P_{Q}(Ax_{n})$ while the  first equation gives the following algorithm: 
\begin{equation}
 x_{n+1}=P_{C}(x_{n}-\lambda A^{*}(I-P_{Q})Ax_{n}),\label{r}
 \end{equation}
Algorithm (\ref{r}) is   exactly the algorithm proposed by Byrne  for more details, see \cite{9} and reference therein.
 
Very recently, Yuan et al., \cite{15}, modified the algorithm of   Moudafi and Al-Shemas \cite{2} and considered the following algorithm: 

\begin{equation}
 {} \left\{ \begin{array}{ll}  
x_{n+1}=(1-\alpha_{n})x_{n}+ \alpha_{n}U(x_{n}-\lambda_{n} A^{*}(Ax_{n}-By_{n}),\\
\\ y_{n+1}=(1-\alpha_{n})y_{n}+ \alpha_{n}T(y_{n}+\lambda_{n} B^{*}(Ax_{n}-By_{n}), \forall n\geq 1, & \textrm{ $  $}
 \end{array}  \right.\label{1.5b}
\end{equation}
where $U,T,$  $A,A^{*},$ $B, B^{*},$ $\lambda_{n},$ $L_{A^{*}A}$ and $L_{B^{*}B}$  as in Algorithm (\ref{1.5}), and $\alpha_{n}\subset[\alpha,1]$ for $\alpha>0.$ By imposing some appropriate  conditions on parameters and the operators involved, they proved a weak convergence result and they also obtained strong convergence result by imposing semicomfactness conditions.
 
In 2015, Chidume et al., \cite{14}  modified Algorithm (\ref{1.5b}) and  considered the following algorithm: 
\begin{equation}
 {} \left\{ \begin{array}{ll} u_{n}= x_{n}-\lambda_{n} A^{*}(Ax_{n}-By_{n})
\\x_{n+1}=(1-\alpha)u_{n}+ \alpha Uu_{n},\\
\\r_{n}=y_{n}+\lambda_{n} B^{*}(Ax_{n}-By_{n})
\\ y_{n+1}=(1-\alpha)r_{n}+ \alpha  Tr_{n}, \forall n\geq 1, & \textrm{ $  $}
 \end{array}  \right.\label{1.5ab}
\end{equation}
where $U,T,$ are two demicontractive mappings, $A,A^{*},$ $B, B^{*},$ $\lambda_{n},$ $L_{A^{*}A}$ and $L_{B^{*}B}$ as in Algorithm (\ref{1.5b}), and $\alpha\in (0,1)$.  Under some appropriate conditions, they also proved a weak convergence result and  strong convergence follows only if  $U,$ and $T$ are semi-compacts.

To solve Problem (\ref{1.2}), Chen et al.,\cite{8} introduced the following Ishikawa extra-gradient iterative methods which generate a sequence $\{x_{n}\}$ by: 
\begin{equation}
 {} \left\{ \begin{array}{ll} x_{0}\in C~ {\rm ~chosen~arbitrarily,}
\\ y_{n}=P_{C}(x_{n}-\lambda_{n} A^{*}(I-UP_{Q})Ax_{n}),
\\ z_{n}=P_{C}(x_{n}-\lambda_{n} A^{*}(I-UP_{Q})Ay_{n}),
\\ w_{n}=(1-\beta_{n})z_{n}+\beta_{n}Tz_{n},
\\ x_{n+1}=(1-\alpha_{n})z_{n}+\alpha_{n}Tw_{n}, \forall n\geq 0, & \textrm{ $  $}
 \end{array}  \right.\label{chen}
\end{equation}
where $\lambda_{n}\subset (0,\frac{1}{2\|A\|^{2}})$ and $\beta_{n},\alpha_{n}\subset (0,1)$ such that $0<a<\beta_{n}<c<\alpha_{n}<\frac{1}{\sqrt{1+L^{2}+1}},$ $U$ is a nonexpansive mapping and  $T$ is L-Lipschitzian pseudocontractive mapping.

Motivated and inspired by the work of; Moudafi \cite{1}, Moudafi and Al-Shemas \cite{2},  Chen et al., \cite{8},  Byrne \cite{9}, Yuan et al., \cite{15} and Chidume et al., \cite{14},    we further propose the following algorithm to solve  the split feasibility and fixed point equality problems (\ref{1.4}) in the case where $U$ and $T$ are quasi-nonexpansive mappings.
\begin{equation}
 {} \left\{ \begin{array}{ll} x_{1}\in H_{1} {\rm ~and~}  x_{2}\in H_{2};
\\z_{n}=P_{C}(x_{n}-\lambda_{n} A^{*}(Ax_{n}-By_{n})),
\\ w_{n}=(1-\beta_{n})z_{n}+\beta_{n}U(z_{n}),
\\ x_{n+1}=(1-\alpha_{n})z_{n}+\alpha_{n}U(w_{n}),
\\
\\u_{n}=P_{Q}(y_{n}+\lambda_{n} B^{*}(Ax_{n}-By_{n})),
\\ r_{n}=(1-\beta_{n})u_{n}+\beta_{n}T(u_{n}),
\\ y_{n+1}=(1-\alpha_{n})u_{n}+\alpha_{n}T(r_{n}),  \forall n\geq 1, & \textrm{ $  $}
 \end{array}  \right.
\end{equation}
where $0<a<\beta_{n}<1,$ $0<b<\alpha_{n}<1,$ and $\lambda_{n}\in\left(0, \frac{2}{L_{1}+L_{2}}\right),$ where $L_{1}$ and $L_{2}$ denote the spectral radius of the operators $A^{*}A$ and $B^{*}B,$ respectively.

It is important to know that the class of quasi-nonexpansive mapping generalizes the class of firmly  quasi-nonexpansive mappings studied by   Moudafi and Al-Shemas \cite{2}. Under some appropriate  conditions  imposed on the parameters and operators  involved, we proved a weak convergence results of the proposed algorithms. Furthermore, we gave numerical example that illustrate our theoretical results. The results presented in this paper, improve, extend and generalize a number of well-known results annouced.

\section{Preliminaries}
In this section, we present some definitions and lemmas  which will be use in proving our main result.

Let $H$ be a Hilbert space and  $T:H\to H$ be a map with $Fix(T)=\{x\in H: Tx=x\}\neq\emptyset.$ $T$  is said to be;
 nonexpansive, if 
$$\left\|Tx-Ty\right\|\leq \left\|x-y\right\|, \forall x,y\in H,$$
 quasi-nonexpansive, if 
\begin{equation}
\left\|Tx-q\right\|\leq \left\|x-q\right\|, \forall x\in H {\rm~and~} q\in Fix(T),\nonumber
\end{equation}
 firmly quasi-nonexpansive, if
\begin{equation}
\left\|Tx-q\right\|^{2}\leq\left\|x-q\right\|^{2}-\left\|Tx-x\right\|^{2}, \forall x\in H {\rm~and~} q\in Fix(T).\nonumber
\end{equation}
And also $T$ is said to be  demiclosed at 0, if for any sequence $\{x_{n}\}$ in $H$ such that  $x_{n}$ converges weakly to $x$ and  $Tx_{n}$ converges strongly to 0, then it implies that $Tx=0.$ And it is said to be semi-compact, if for any bounded sequence $\{x_{n}\}\subset H$ with $(I-T)x_{n}$ converges strongly  to 0,  there exists a sub-sequence say  $\{x_{n_{k}}\}$ of $\{x_{n}\}$ such that $\{x_{n_{k}}\}$ converges strongly to 0.

\begin{lemma}\label{opial}{\rm(Opial~ \cite{17})} Let $H$ be a real Hilbert space and $\{x_{n}\}$ be a
sequence in $H$ such that there exists a nonempty set $C\subset H$ such that the following 
conditions are satisfied:
\begin{enumerate}
\item[(i)] For each $x\in C$, $\underset{n\to\infty}{\lim}{\|x_{n}-x\|}$ exists,
\item[(ii)] Any weak-cluster point of the sequence $\{x_{n}\}$ belongs to $C.$
\end{enumerate}
Then, there exists $y\in C$ such that $\{x_{n}\}$  converges weakly to $y.$
\end{lemma}
In sequel, adopt the following notations: 
\begin{itemize} 
\item [(i)]$I:$ The identity operator on a Hilbert space $H,$
\item [(ii)]$Fix(T):$  The fixed point set of $T$ i.e., $Fix(T)=\{x\in H: Tx=x\},$ 
\item[(iii)] $"\rightarrow "$and $"\rightharpoonup"$ The strong and weak covergence, respectively,
\item [(iv)]$\omega_{\omega}(x_{n}):$ The set of the cluster point of $\{x_{n}\}$ in the weak topology i.e., $\{{\rm~there~ exists~} \{x_{n_{k}}\}$ of $\{x_{n}\}$  such that $x_{n_{k}}\rightharpoonup x\},$ 
\item [(v)]$\Omega:$ The  solution set of  Problem (\ref{1.4}), i.e., 
\begin{equation}
\Omega=\Big\{{\rm~Find~} x^{*}\in C\cap Fix(U) {\rm~and~} y^{*}\in Q\cap Fix(T) {\rm~such~ that~} Ax^{*}=By^{*}\Big\}.\label{2.2}
\end{equation}
 \end{itemize}
\section{Main Results}
To approximate the solution of  split feasibility and fixed point equality problems (\ref{2.2}), we make the following assumptions:
\begin{enumerate}
\item [($B_{1}$)]$H_{1},$ $H_{2},$ $H_{3},$ are real Hilbert spaces, $C$ and $Q$ are two nonempty closed convex subset of $H_{1}$ and  $H_{2},$ respectively. 
\item [($B_{2}$)] $U:H_{1}\to H_{1}$ and  $T:H_{2}\to H_{2}$ are two quasi-nonexpansive mappings with  $Fix(U)\neq\emptyset$ and
$Fix(T)\neq\emptyset.$ 
\item [($B_{3}$)]  $A:H_{1}\to H_{3}$ and $B:H_{2}\to H_{3}$  are two bounded linear operators with their adjoints $A^{*}$ and $B^{*},$ respectively.
\item [($B_{4}$)] $(U-I)$ and $(T-I)$ are demiclosed at zero.
\item [($B_{5}$)] $P_{C}$ and $P_{Q}$ are metric projection of $H_{1}$ and  $H_{2}$ onto $C$ and $Q,$ respectively.
\item [($B_{6}$)] For arbitrary $x_{1}\in H_{1}$ and $ y_{1}\in H_{2},$ define a sequence $\{(x_{n}, y_{n})\}$ by:
\end{enumerate}

\begin{equation}
 {} \left\{ \begin{array}{ll} 
z_{n}=P_{C}(x_{n}-\lambda_{n} A^{*}(Ax_{n}-By_{n})),
\\ w_{n}=(1-\beta_{n})z_{n}+\beta_{n}U(z_{n}),
\\ x_{n+1}=(1-\alpha_{n})z_{n}+\alpha_{n}U(w_{n}),
\\
\\u_{n}=P_{Q}(y_{n}+\lambda_{n} B^{*}(Ax_{n}-By_{n})),
\\ r_{n}=(1-\beta_{n})u_{n}+\beta_{n}T(u_{n}),
\\ y_{n+1}=(1-\alpha_{n})u_{n}+\alpha_{n}T(r_{n}),  \forall n\geq 1. & \textrm{ $  $}
 \end{array}  \right.\label{BUL}
\end{equation}
where $0<a<\beta_{n}<1,$ $0<b<\alpha_{n}<1,$ and $\lambda_{n}\in\left(0, \frac{2}{L_{1}+L_{2}}\right),$ where $L_{1}$ and $L_{2}$ denote the spectral radius of the operators $A^{*}A$ and $B^{*}B,$ respectively.

We are now in the position to state and prove our main result.
\begin{theorem}\label{T1} Suppose that  assumptions $(B_{1})-(B_{6})$ are satisfied, in addition assume that the solution set $\Omega\neq\emptyset.$  Then, the sequence $\{(x_{n}, y_{n})\}$ generated by Algorithm (\ref{BUL}) converges weakly to $(x^{*}, y^{*})\Omega.$
\end{theorem}
\begin{proof}
Let $(x^{*}, y^{*})\in\Omega,$ by (\ref{BUL}), we have 
\begin{align}
\left\|x_{n+1}-x^{*}\right\|^{2}&=\left\|(1-\alpha_{n})(z_{n}-x^{*})+\alpha_{n}(Uw_{n}-x^{*})\right\|^{2}\nonumber
\\=&(1-\alpha_{n})\left\|z_{n}-x^{*}\right\|^{2}+\alpha_{n}\left\|Uw_{n}-x^{*}\right\|^{2}-\alpha_{n}(1-\alpha_{n})\left\|Uw_{n}-z_{n}\right\|^{2}\nonumber
\\\leq& (1-\alpha_{n})\left\|z_{n}-x^{*}\right\|^{2} +\alpha_{n}\left\|w_{n}-x^{*}\right\|^{2}-\alpha_{n}(1-\alpha_{n})\left\|Uw_{n}-z_{n}\right\|^{2}.\label{a}
\end{align}
On the other hand,
\begin{align}
\left\|w_{n}-x^{*}\right\|^{2}&=\left\|(1-\beta_{n})(z_{n}-x^{*})+\beta_{n}(Uz_{n}-x^{*})\right\|^{2}\nonumber
\\=&(1-\beta_{n})\left\|z_{n}-x^{*}\right\|^{2}+\beta_{n}\left\|Uz_{n}-x^{*}\right\|^{2}-\beta_{n}(1-\beta_{n})\left\|Uz_{n}-z_{n}\right\|^{2}\nonumber
\\\leq& \left\|z_{n}-x^{*}\right\|^{2}-\beta_{n}(1-\beta_{n})\left\|Uz_{n}-z_{n}\right\|^{2},\label{b}
\end{align}
and 
\begin{align}
\left\|z_{n}-x^{*}\right\|^{2}&=\left\|P_{C}(x_{n}-\lambda_{n} A^{*}(Ax_{n}-By_{n}))-P_{C}(x^{*})\right\|^{2}\nonumber
\\&\leq \left\|x_{n}-\lambda_{n} A^{*}(Ax_{n}-By_{n})-x^{*}\right\|^{2}\nonumber
\\&= \left\|x_{n}-x^{*}\right\|^{2}-2\lambda_{n}\left\langle  Ax_{n}-Ax^{*}, Ax_{n}-By_{n}\right\rangle+\lambda_{n}^{2}L_{1}\left\|Ax_{n}-By_{n}\right\|^{2}\label{c}
\end{align}
From $(\ref{a})-(\ref{c}),$ we obtain that
\begin{align}
\left\|x_{n+1}-x^{*}\right\|^{2}&\leq\left\|x_{n}-x^{*}\right\|^{2}-2\lambda_{n}\left\langle  Ax_{n}-Ax^{*}, Ax_{n}-By_{n}\right\rangle+\lambda_{n}^{2}L_{1}\left\|Ax_{n}-By_{n}\right\|^{2}\nonumber
\\&-\alpha_{n}\beta_{n}(1-\beta_{n})\left\|U(z_{n})-z_{n}\right\|^{2}-\alpha_{n}(1-\alpha_{n})\left\|Uw_{n}-z_{n}\right\|^{2}.\label{d}
\end{align}
Similarly, the second equation of Equation (\ref{BUL}) gives
\begin{align}
\left\|y_{n+1}-y^{*}\right\|^{2}&\leq\left\|y_{n}-y^{*}\right\|^{2}+2\lambda_{n}\left\langle  By_{n}-By^{*}, Ax_{n}-By_{n}\right\rangle+\lambda_{n}^{2}L_{2}\left\|Ax_{n}-By_{n}\right\|^{2}\nonumber
\\&-\alpha_{n}\beta_{n}(1-\beta_{n})\left\|T(u_{n})-u_{n}\right\|^{2}-\alpha_{n}(1-\alpha_{n})\left\|Tr_{n}-u_{n}\right\|^{2}.\label{e}
\end{align}
By (\ref{d}), (\ref{e}) and noticing  that $Ax^{*} = By^{*},$ we deduce that
\begin{align}
\left\|x_{n+1}-x^{*}\right\|^{2}+\left\|y_{n+1}-y^{*}\right\|^{2}&\leq \left\|x_{n}-x^{*}\right\|^{2}+\left\|y_{n}-y^{*}\right\|^{2}-2\lambda_{n}\left\|Ax_{n}-By_{n}\right\|^{2}\nonumber
\\&+\lambda_{n}^{2}(L_{1}+L_{2})\left\|Ax_{n}-By_{n}\right\|^{2}\nonumber
\\&-\alpha_{n}\beta_{n}(1-\beta_{n})\left\|U(z_{n})-z_{n}\right\|^{2}\nonumber
\\&-\alpha_{n}\beta_{n}(1-\beta_{n})\left\|T(u_{n})-u_{n}\right\|^{2}.\label{f}
\end{align}
Thus, we deduce that
\begin{align}
\Omega_{n+1}&\leq \Omega_{n}-\lambda_{n}\left(2-\lambda_{n}^{2}(L_{1}+L_{2})\right)\left\|Ax_{n}-By_{n}\right\|^{2}\nonumber
\\&-\alpha_{n}\beta_{n}(1-\beta_{n})\left\|U(z_{n})-z_{n}\right\|^{2}-\alpha_{n}\beta_{n}(1-\beta_{n})\left\|T(u_{n})-u_{n}\right\|^{2},\label{g}
\end{align}
where $$\Omega_{n}:=\left\|x_{n}-x^{*}\right\|^{2}+\left\|y_{n}-y^{*}\right\|^{2}.$$
Thus, $\{\Omega_{n}\}$ is a non-increasing sequence and bounded below by 0, therefore, it converges.

From (\ref{g}) and the fact that $\{\Omega_{n}\}$ converges, we deduce that
\begin{align}
\underset{n\to\infty}{\lim}\left\|Ax_{n}-By_{n}\right\|=0,\label{q}
\end{align}
\begin{align}
\underset{n\to\infty}{\lim}\left\|Uz_{n}-z_{n}\right\|=0 {\rm~and~} \underset{n\to\infty}{\lim}\left\|Tu_{n}-u_{n}\right\|=0\label{K}. 
\end{align}

Furthermore, since $\{\Omega_{n}\}$ converges, this ensures that  $\{x_{n}\}$ and $\{y_{n}\}$  also converges.\\ 

Now, let $(x,y)\in\Omega,$ this implies that $x\in C\cap Fix(U)$ and $y\in Q\cap Fix(T)$ such that $Ax=By.$ 

 The fact that $x_{n}\rightharpoonup x$ and $\underset{n\to\infty}{\lim}\left\|Ax_{n}-By_{n}\right\|=0$ together with $$z_{n}=P_{C}(x_{n}-\lambda_{n} A^{*}(Ax_{n}-By_{n})),$$
we deduce that   $z_{n}\rightharpoonup P_{C}x.$  Since $x\in C,$ by projection theorem, we obtain that $P_{C}x=x.$ Hence,  $z_{n}\rightharpoonup x.$

Similarly, The fact that $y_{n}\rightharpoonup y$ and $\underset{n\to\infty}{\lim}\left\|Ax_{n}-By_{n}\right\|=0$ together with 
$$u_{n}=P_{Q}(y_{n}+\lambda_{n} B^{*}(Ax_{n}-By_{n})),$$
we deduce that   $u_{n}\rightharpoonup P_{Q}x.$  Since $x\in Q,$ by projection theorem, we obtain that $P_{Q}y=y.$ Hence,  $u_{n}\rightharpoonup y.$



Now, $z_{n}\rightharpoonup x$ and $\underset{n\to\infty}{\lim}\left\|Uz_{n}-z_{n}\right\|=0$ together with the demiclosed of $(U-I)$ at zero, we deduce that $x\in Fix(U)$ which implies that $x\in Fix(U).$ 

On the other hand, $u_{n}\rightharpoonup y$ and $\underset{n\to\infty}{\lim}\left\|Tu_{n}-u_{n}\right\|=0$ together with the demiclosed of $(T-I)$ at zero, we deduce that $y\in Fix(T)$ which implies that $y\in Fix(T).$  

Since $z_{n}\rightharpoonup x,$  $u_{n}\rightharpoonup y$ and the fact that $A$ and $B$ are bounded linear operators, we have 
$$Az_{n}\rightharpoonup Ax {\rm~~and~~} Bu_{n}\rightharpoonup By,$$
This implies that 
$$Az_{n}-Bu_{n}\rightharpoonup Ax-By,$$
which turn to implies that 
$$\left\|Ax-By\right\|\leq\underset{n\to\infty}{\liminf}\left\|Az_{n}-Bu_{n}\right\|=0,$$
which further implies that $Ax=By.$ Noticing that $x\in C,$  $x\in Fix(U),$ $y\in Q$ and $y\in Fix(T)$, we have that $x\in C\cap Fix(U)$ and $y\in Q\cap Fix(T).$ Hence, we conclude that $(x,y)\in \Omega.$ 

Summing up, we have proved that:

\begin{enumerate}
\item[(i)] for each $(x^{*}, x^{*})\in\Omega,$   the $\underset{n\to\infty}{\lim}\left(\left\|x_{n}-x^{*}\right\|^{2}+\left\|y_{n}-y^{*}\right\|^{2}\right)$ exists;
\item[(ii)] each weak cluster of the sequence $(x_{n}, y_{n})$ belongs to $\Omega.$
\end{enumerate}
Thus, by Lemma (\ref{opial}) we conclude that the  sequences $(x_{n}, y_{n})$ converges weakly to $(x^{*}, x^{*})\in\Omega.$ And the proof is complete.
 \end{proof}

\begin{theorem}\label{T2}Suppose that all the hypothesis of Theorem \ref{T1} are satisfied and in addition,  $U$ and $T$ are semi-compacts, then, the sequence $\{(x_{n}, y_{n})\}$ converges strongly to  $(x^{*}, y^{*})\in\Omega.$ 
\end{theorem}
\begin{proof}
As in the proof of Theorem \ref{T1}, $\{u_{n}\}$ and $\{z_{n}\}$ are bounded, by (\ref{K}) and the fact that $U$ and $T$ are semi-compacts, then there exists a sub-sequences    $\{u_{n_{k}}\}$ and $\{z_{n_{k}}\}$ (suppose without loss of generality) of $\{u_{n}\}$ and $\{z_{n}\}$ such that $u_{n_{k}}\to x$ and $z_{n_{k}}\to y.$ Since, $u_{n}\rightharpoonup x^{*}$ and $z_{n}\rightharpoonup y^{*}$, we have $ x=x^{*}$ and $y= y^{*}.$ By (\ref{q}) and the fact that $u_{n_{k}}\to x^{*}$ and $z_{n_{k}}\to y^{*},$ we have
\begin{align}
\underset{n\to\infty}{\lim}\left\|Ax^{*}-Ay^{*}\right\|=\underset{n\to\infty}{\lim}\left\|Au_{n_{k}}-Bz_{n_{k}}\right\|=0.
\end{align}
which turn to implies that $Ax^{*}=Ay^{*}$. Hence $(x^{*}, y^{*})\in\Omega$. Thus, the iterative algorithm of Theorem \ref{T1} conveges strongly to the solution of Problem \ref{2.2}.
\end{proof}

\begin{cor}
Suppose  that  conditions $(B_{1})-(B_{6})$ are satisfied and let the sequence $\{(x_{n}, y_{n})\}$ be generated  by Algorithm (\ref{BUL}). 
Assume that $\Omega\neq\emptyset$ and let $U$ and $T$ be a firmly of quasi-nonexpansive mappings.
Then, the sequence  $\{(x_{n}, y_{n})\}$  generated  by Algorithm (\ref{BUL}) converges weakly to the solution set of Problem (\ref{2.2}).
\end{cor}

\begin{cor}
Suppose  that  conditions $$(B_{1})-(B_{5})$$ are satisfied are satisfied and let the sequence $\{(x_{n}, y_{n})\}$ be generated  by 

\begin{equation}
 {} \left\{ \begin{array}{ll} 
z_{n}=x_{n}-\lambda_{n} A^{*}(Ax_{n}-By_{n}),
\\  x_{n+1}=(1-\alpha_{n})z_{n}+\alpha_{n}U(z_{n}),
\\
\\u_{n}=y_{n}+\lambda_{n} B^{*}(Ax_{n}-By_{n}),
\\ y_{n+1}=(1-\alpha_{n})u_{n}+\alpha_{n}T(_{n}),  \forall n\geq 0. & \textrm{ $  $}
 \end{array}  \right.\label{BUL3}
\end{equation}
 where $0<a<\beta_{n}<1,$  and $\lambda_{n}\in\left(0, \frac{2}{L_{1}+L_{2}}\right),$ where $L_{1}$ and $L_{2}$ denote the spectral radius of the operators $A^{*}A$ and $B^{*}B,$ respectively. Assume that $\Omega\neq\emptyset.$ 
Then, the sequence  $\{(x_{n}, y_{n})\}$  generated  by Algorithm (\ref{BUL3}) converges weakly to the solution  of SEFPP (\ref{1.3}).
\end{cor}
\begin{proof}
Trivially, Algorithm (\ref{BUL}) reduces to Algorithm (\ref{BUL3}) as $\beta=0,$ $P_{C}=P_{Q}=I$ and SFFPEP (\ref{1.4}) reduces to SEFPP (\ref{1.3}) as $C:=Fix(U)$ and $Q:=Fix(T).$ Therefore, all the hypothesis of Theorem  \ref{T1} are satisfied. Hence, the proof of this corollary follows directly from Theorem \ref{T1}. 
\end{proof}

\section{Numerical Example}
In this section, we give a numerical example  to illustrate our theoretical results. 
\begin{example}\label{example}Let $H_{1}=\Re$ with the inner product defined by $\left\langle x, y\right\rangle=xy$ for all $x, y\in \Re$ and  $\|.\|$ stand for the corresponding norm. Let $C:=[0,\infty),$  $Q:=[0,\infty)$ and define a mappings $T:C\to \Re$  and  $S:Q\to \Re$ by 
$Tx = \frac{x^{2}+5}{1+x},$ for all $x\in C$ and $Sx=\frac{x+5}{5}$, for all $x\in Q.$ Then $T$ and $S$ are quasi nonexpansive mappings.
\end{example}
\begin{proof}Trivially, $Fix(T)=5 $ and $Fix(S)=\frac{5}{4}.$  

Now, \begin{align*}\left|Tx-5 \right|&=\left|\frac{x^{2}+5}{1+x}-5\right|=\frac{x}{1+x}\left|x-5\right|
\\&\leq \left|x-5\right|.
\end{align*}

On the other hand,
\begin{align*}\left|Sx-\frac{5}{4} \right|&=\left|\frac{x+5}{5}-\frac{5}{4}\right|=\frac{1}{5}\left|x-\frac{5}{4}\right|
\\&\leq \left|x-5\right|.
\end{align*}
Hence, $T$ and $S$ are quasi-nonexpansive mappings. 
\end{proof}

\begin{example}Let $H_{1}=\Re$ with the inner product defined by $\left\langle x, y\right\rangle=xy$ for all $x, y\in \Re$ and  $\|.\|$ stand for the corresponding norm. Let $C:=[0,\infty),$  $Q:=[0,\infty)$ and define a mappings $U:C\to \Re$  and  $T:Q\to \Re$ by 
$Ux = \frac{x^{2}+5}{1+x},$ for all $x\in C$ and $Tx=\frac{x+5}{5}$, for all $x\in Q.$ And also let $P_{C}=P_{Q}=I,$  $Ax=x,$ $By=4y,$ $\lambda_{n}=1$, $\alpha_{n}=\frac{1}{5},$ $\beta_{n}=\frac{1}{8}$ and  $\{(x_{n},y_{n})\}$ be the sequence  generated by

\begin{equation}
 {} \left\{ \begin{array}{ll} x_{0}\in C{\rm~~and~~} y_{0}\in Q,
\\z_{n}=P_{C}(x_{n}- A^{*}(x_{n}-4y_{n})),
\\ w_{n}=(1-\frac{1}{8})z_{n}+\frac{1}{8}U(z_{n}),
\\ x_{n+1}=(1-\frac{1}{5})z_{n}+\frac{1}{5}U(w_{n}),
\\
\\u_{n}=P_{Q}(y_{n}+ B^{*}(x_{n}-4y_{n})),
\\ r_{n}=(1-\frac{1}{8})u_{n}+\frac{1}{8}T(u_{n}),
\\ y_{n+1}=(1-\frac{1}{5})u_{n}+\frac{1}{5}T(r_{n}),  \forall n\geq 0. & \textrm{ $  $}
 \end{array}  \right.\label{BUL11}
\end{equation}
Then, $\{(x_{n},y_{n})\}$ converges  to  $(5,5/4)\in\Omega$.
\end{example}
\begin{proof}By  Example \ref{example} $U$ and $T$ are quasi-nonexpansive mappings. Clearly,   $A$ and $B$ are  bounded linear operator on  $\Re$  with $A=A^{*}=1$ and $B=B^{*}=4,$ respectively. Furthermore, it is easy to see that $Fix(U)=5$ and $Fix(T)=\frac{5}{4}.$ Hence, 
\begin{equation*}
\Omega=\Big\{ 5\in C\cap Fix(U) {\rm~and~} 5/4\in Q\cap Fix(T)  {\rm~such ~that~} A(5)=B(5/4)\Big\}. 
\end{equation*}

After simplification, Algorithm (\ref{BUL11}) reduces to
\begin{equation}
 {} \left\{ \begin{array}{ll} x_{0}\in C{\rm~~and~~} y_{0}\in Q,
\\z_{n}=x_{n},
\\ w_{n}=\frac{7}{8}z_{n}+\frac{1}{8}(\frac{z_{n}^{2}+5}{z_{n}+1}),
\\ x_{n+1}=\frac{4}{5}z_{n}+\frac{1}{5}(\frac{w_{n}^{2}+5}{w_{n}+1}),
\\
\\u_{n}=y_{n},
\\ r_{n}=\frac{7}{8}u_{n}+\frac{1}{8}(\frac{u_{n}+5}{5}),
\\ y_{n+1}=\frac{4}{5}u_{n}+\frac{1}{5}(\frac{r_{n}+5}{5}),  \forall n\geq 0. & \textrm{ $  $}
 \end{array}  \right.
\end{equation}

\begin{table}[htbp]
	\begin{center}
		\caption{Starting with initial values $x_{0}=10$ and $y_{0}= 15$}
	\label{fig1}
		\begin{tabular}{|c|c|c|} 
			\hline
			n & $x_{n}$&$y_{n}$\\ [0.5ex] 
			\hline
			0 & 10.00000000&  15.00000000\\
			1 & 9.898293685&  12.74500000\\ 
			
			2 & 9.797736851 &10.85982000\\
			
			3 & 9.698337655 & 9.283809520\\
		
			. &.   &. \\
			
			. &.  &. \\
			
			. &.	& . \\
			
			248 &5.001051418&1.250000002 \\
			
			249 & 5.001012726	 & 1.250000002\\
			
			250 & 5.000975458	&  1.250000002	   \\ [1ex] 
			\hline
		\end{tabular}
	\end{center}
	\end{table}

\begin{figure}[htbp]
\caption{The convergence of $\{(x_{n}, y_{n})\}$ with the initial value $x_{0}=10$ and $y_{0}=15$}
	\label{fig2}
	\centering
		\includegraphics[scale=0.5]{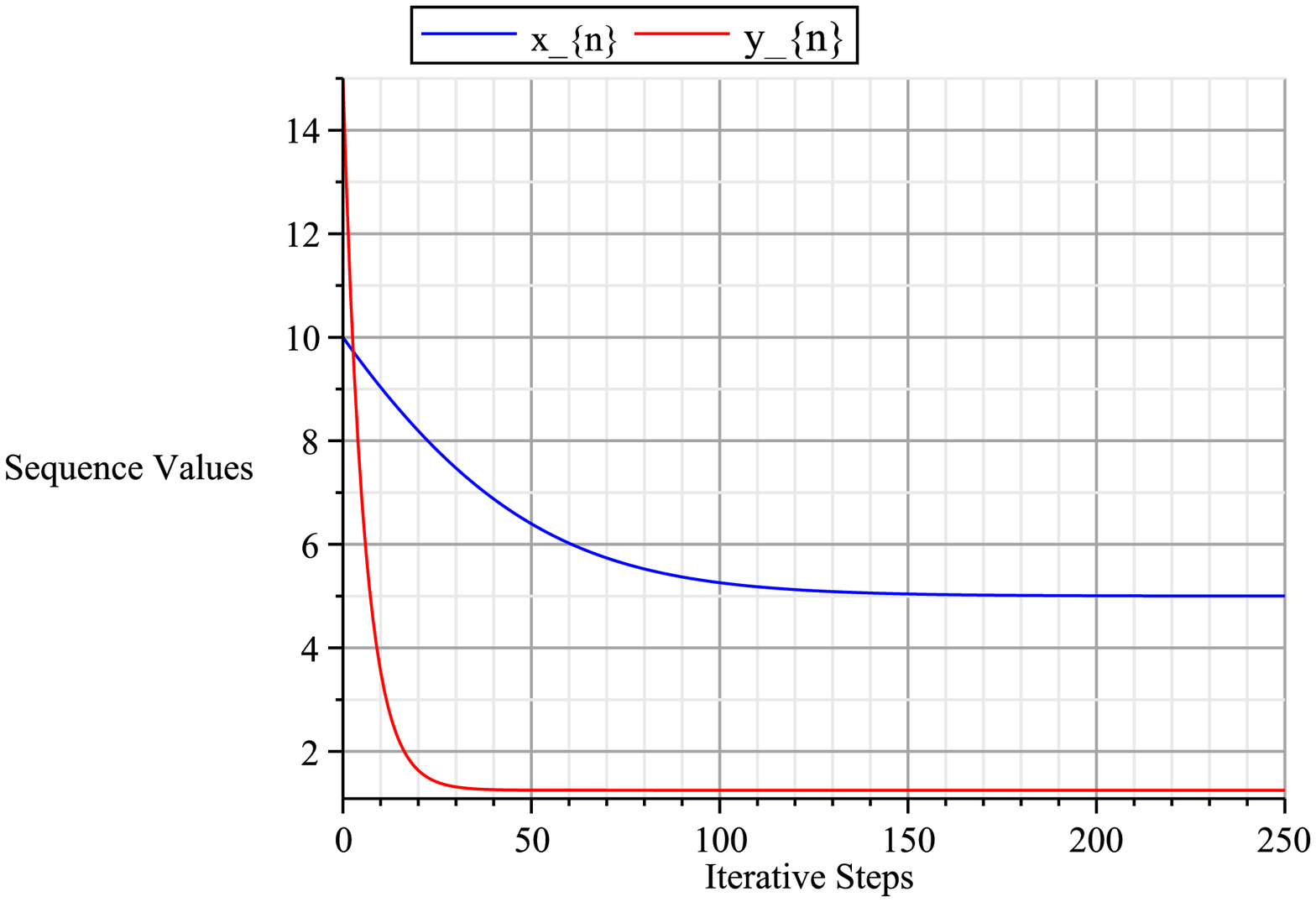}
		\end{figure}

\begin{table}[htbp]
	\begin{center}
		\caption{Starting with initial values $x_{0}=5$ and $y_{0}= 1.25$}
	\label{tab:StartingWithInialValue101}
		\begin{tabular}{|c|c|c|} 
			\hline
			n & $x_{n}$&$y_{n}$\\ [0.5ex] 
			\hline
			0 & 5.000000000&  1.250000000\\ 
			
			1 & 5.000000000 &1.250000000\\
			
			2 & 5.000000000 & 1.250000000\\
		
			. &.   &. \\
			
			. &.  &. \\
			
			. &.	& . \\
			
			98 &5.000000000&1.250000000 \\
			
			99 & 5.000000000	 & 1.250000000\\
			
			100 & 5.000000000	&  1.250000000	   \\ [1ex] 
			\hline
		\end{tabular}
	\end{center}
	\end{table}
\begin{figure}[htbp]
\caption{The convergence of $\{(x_{n}, y_{n})\}$ with the initial value $x_{0}=5$ and $y_{0}=1.25$}
	\label{fig4}
	\centering
		\includegraphics[scale=0.5]{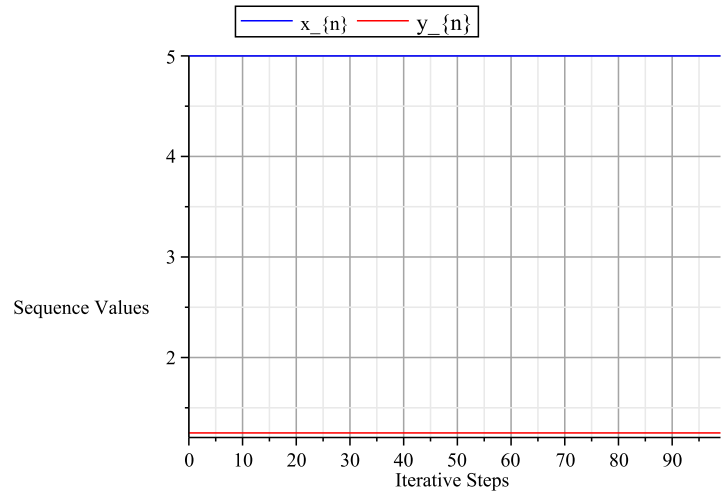}
		\end{figure}
\end{proof}

\section{Conclusion}
In this paper,  we introduce a new problem called split feasibility and fixed point equality problems (SFFPEP) and study it for the class of quasi-nonexpansive mappings in  Hilbert spaces. Under some suitable assumptions imposed on the parameters and operators involved, we proved a weak convergence theorem of the propose problem. Furthermore,  we gives a numerical example that illustrate our theoretical result.  The results presented in this paper,  extend and complement the results of;   Moudafi \cite{1}, Moudafi and Al-Shemas \cite{2},  Chen et al., \cite{8},  Byrne \cite{9}, Yuan et al., \cite{15} and Chidume et al., \cite{14}.  

The split feasibility and fixed point equality problem (SFFPEP) is a very interesting topic. Its generalizes the split feasibility  problem  (SFP), fixed point problem (FPP), split feasibilty and fixed point problem (SFFPP) and split equality fixed point problem (SEFPP) . All the results and conclusions that are true for the split feasibility and fixed point equality problem (SFFPEP) continue to holds for these problems (SFP,FPP,SFFPP and SEFPP) and it   definitely shows the significance and the range of applicability of split feasibility and fixed point equality problem (SFFPEP).

\begin{remark}
Theorem \ref{T2} gives a strong convergence result for the class of quasi-nonexpansive mappings with the assumption that each mapping is a semi-compact. This compactness type condition  appeared very strong as only few mapping are semi-compact.  

This leads us to think of the following question: 
\begin{enumerate}
 \item [(i)] Can the strong convergence of Theorem \ref{T1} be obtain  without imposing the semi-compactness   conditions? 
\item [(ii)]If the above answer is affirmative, can the strong convergence hold for the class of  infinite family of quasi-nonexpansive mappings?
\end{enumerate}
This will be our future research.
\end{remark}

\end{document}